\documentclass[a4paper, 11pt]{article}
\usepackage{verbatim}

\usepackage{appendix}
\usepackage{amsmath,amssymb, amsthm}
\newtheorem{theorem}{Theorem}[section]

\newtheorem{corollary}[theorem]{Corollary}

\newtheorem{lemma}[theorem]{Lemma}

\newtheorem{proposition}[theorem]{Proposition}
\newtheorem{remark}[theorem]{Remark}

\topmargin by -\headsep
\title{Principal submatrices, restricted invertibility and a quantitative Gauss-Lucas theorem}
\begin{document}
\author{ Mohan Ravichandran\footnote{email : mohan.ravichandran@msgsu.edu.tr, supported by Tubitak 115F204}}
\date{}
\maketitle

\begin{abstract}
We apply the techniques developed by Marcus, Spielman and Srivastava, working with principal submatrices in place of rank $1$ decompositions to give an alternate proof of their results on restricted invertibility.  We show that one can find well conditioned column submatrices all the way upto the so called modified stable rank. All constructions are algorithmic. A byproduct of these results is an interesting quantitative version of the classical Gauss-Lucas theorem on the critical points of complex polynomials.  We show that for any degree $n$ polynomial $p$ and any $c \geq \frac{1}{2}$, the area of the convex hull of the roots of $p^{(cn)}$ is at most $4(c-c^2)$ that of the area of the convex hull of the roots of $p$. 

\end{abstract}
\section{Introduction}
The Bourgain-Tzafriri restricted invertibility principle \cite{BT87,BT91} is a fundamental fact about the existence of well conditioned restrictions of a linear operator between finite dimensional spaces. Recall that given a linear operator on $\mathbb{C}^n$, the stable rank, which we denote $\operatorname{srank}(T)$ is defined to be the quantity,
\[\operatorname{srank}(T) = \dfrac{||T||^2_2}{||T||^2}.\]
Here, we use the operator algebraic convention of using $||T||$ to represent the operator norm, that is, the largest singular value of $T$ and $||T||_2$ to be the Frobenius or the Hilbert-Schmidt norm, $(\sum s_k(T)^2)^{1/2}$, where the $s_k$ are the singular values of $T$. The restricted invertibility principle says the following, 

\begin{theorem}[Bourgain-Tzafriri]
There are universal constants $0 < c< 1$ and $0 < d < 1$ such that for any $T \in M_n(\mathbb{C})$ such that $||Te_i|| = 1$, for $i \in [n]$, we can find a subset $\sigma \subset [n]$ such that 
\[|\sigma| > \dfrac{cn}{||T||^2} = c \,\operatorname{srank}(T)\] and such that 
\[\lambda_{\operatorname{min}}(TP_{\sigma}\mid_{P_{\sigma}\mathbb{C}^n})> d \]
\end{theorem}

Spielman and Srivastava  \cite{SSEle} gave a remarkable proof of this theorem in 2012, that had the triple merits of being transparent, constructive and providing tight constants.  They showed that one may take $c = \epsilon^2$ and $d = (1-\epsilon)$. There is a related theorem of Bourgain and Tzafriri \cite{BT91} which says,
\begin{theorem}\label{Thm2}
There are universal constants $c$ and $\epsilon$ such that for any $T \in M_n(\mathbb{C})$ with $\Delta(T) = 0$ and $||T|| \leq 1$, there is a subset $\sigma \subset [n]$ such that 
\[|\sigma| > cn,\]
and
\[||P_{\sigma} T P_{\sigma}|| < \epsilon.\]
\end{theorem}

Assaf Naor asked in 2010 \cite{NaorBourbaki} if there was a constructive solution to this problem, parallel to the Spielman-Srivastava result for restricted invertibility. Pierre Youssef answered this question in 2012 \cite{YouCSIMRN} and provided constants, for hermitian matrices, of $c = \left[\dfrac{(\sqrt{2}-1)^2}{\sqrt{2}}\epsilon \right]^2 \approx 0.015 \,\epsilon^2$  for any $\epsilon < 1$. 

These theorems are subsumed by the main theorem in the work of Marcus, Spielman and Srivastava \cite{MSS2} that yields the solution to the Kadison-Singer problem as a corollary. Remarkably, the constants in the solution to the Kadison-Singer problem are nearly identical to the constants in the restricted invertibility principle. The MSS theorem however, is non-constructive and finding a constructive way of constructing the partitions it guarantees is currently, a major open problem. Marcus, Spielman and Srivasta, in the third of their series of papers on interlacing polynomials \cite{MSS3}, prove improved restricted invertibility estimates, see remark(\ref{MSSIRI}) for a discussion. Our results are similar to theirs but we hope that this slightly different approach will prove useful in other settings as well. Assaf Naor and Pierre Youssef in their recent paper \cite{NaoYou} also proved improved estimates for restricted invertibility, using a variety of deep tools from geometric functional analysis. Their estimates are however of a slightly different flavour from those of MSS in \cite{MSS3}. 

Spielman and Srivastava's constructive proof of the restricted invertibility principle \cite{SSEle} can be applied to give a proof of theorem \ref{Thm2}. One can conclude that in the case that the matrix $T$ in question is additionally,  hermitian, that one can effectively find a large subset $\sigma$ such that $P_{\sigma}TP_{\sigma} < \epsilon\,I$. Repeating the argument for this submatrix, but instead bounding the smallest eigenvalue, one can get norm bounds for principal submatrices of hermitian matrices. The one issue with this approach is that the dependance of $c$ on $\epsilon$ is suboptimal, yielding $c = O(\epsilon^4)$ rather than the theoretically optimal $c = O(\epsilon^2)$. 

The paper of Marcus, Spielman and Srivastava, \cite{MSS2}, see also \cite{MSSICM}, where they solve the Kadison-Singer problem clarifies further, the earlier proof of the restricted invertilibility theorem due Spielman and Srivastava. They show how this earlier proof can be understood in the framework of their method of interlacing polynomials. This method reduces the study of restricted invertilibilty to studying how the operation on real rooted polynomials $f \longrightarrow \left(1-\dfrac{d}{dx}\right)f$ affects the smallest strictly positive root. Another remarkable method, the so called barrier method, also due to Marcus, Spielman and Srivastava can be effectively used for this last problem. 

The key observation in this paper is that the method of interlacing polynomials can be directly applied to hermitian matrices and their principal submatrices. Let $A$ be a hermitian matrix in $M_n(\mathbb{C})$ and let $A_k$, for $k \in [n]$ be the principal submatrix derived by deleting the $k$'th row and column of $A$. 

The celebrated Cauchy interacing principle says that the eigenvalues of $A$ and $A_k$ interlace for each $k \in [n]$. Alternately, the characteristic polynomials $\chi[A_k]$ have $\chi[A]$ as a common interlacer. It is a lovely fact, first observed by R.C.Thompson \cite{ThoPS1} that 
\[\sum_{k\in[n]} \chi[A_k] = \chi'[A].\]

This observation allows us to immediately use the machinery of interlacing polynomials of MSS to detect principal submatrices whose largest eigenvalue is small. Given $S \subset [n]$, Let $A_S$ denote the principal submatrix of $A$ created by removing the rows and columns corresponding to elements in $S$. The first theorem in this paper is, 
\begin{theorem} Given a hermitian matrix $A \in M_n(\mathbb{C})$ and $k \leq n$, one can always find a subset $S \subset [n]$ with $|S| =k$ such that the principal submatrix $A_S$ satisfies
\begin{eqnarray}\label{interlacing}
\lambda_{max}(A_S) \leq \operatorname{max root}\,\chi^{(n-k)}[A].
\end{eqnarray}
\end{theorem}

Remarkably, this quantity can be well controlled. This might a priori seem intractable as polynomial $\chi[A]$ could be any real rooted polynomial at all. However, the barrier method of Batson, Spielman and Srivastava \cite{BSS} naturally applies to the problem of bounding the roots of the higher derivatives of a polynomial. With very little effort, one can use their technique to prove the following result which seems to be known and was mentioned by Adam Marcus in a talk in 2014,

\begin{theorem}[Marcus, 2014]
Let $p$ be a real rooted polynomial with roots all less than $1$. Then, for any $k \leq n$, we have
\[\operatorname{max root} p^{(k)} \leq \operatorname{inf}_{b \geq 1} \,\,b - \dfrac{k}{\varphi(b)},\qquad \text{where } \varphi(b) := \sum \dfrac{1}{b - \lambda_i},\]
and where the $\lambda_i$ are the roots of $p$. 
\end{theorem}

Optimizing the above is routine. We mention here one consequence, that I personally find absolutely marvellous. Given any real rooted polynomial $p$ of degree $n$ with all roots lying in $[-1,1]$ and the average of the roots $0$ and any $c \geq \dfrac{1}{2}$, we have that 
\[\operatorname{roots}(p^{(cn)}) \subset [-2\sqrt{c-c^{2}}, 2\sqrt{c-c^{2}}].\] 

 The necessity of having to take $c \geq \dfrac{1}{2}$ can be understood by looking at the polynomial $[(x+1)(x-1)]^{n/2}$. Taking the $\dfrac{n}{2} -1$'th derivative still leaves a root at $1$ and that one needs to take further derivatives to push the largest root inward. Combining this estimate with (\ref{interlacing}), we conclude the following,

\begin{theorem}\label{zd1}
Let $A \in M_n(\mathbb{C})$ be hermitian and such that  $\operatorname{Tr}(A) = 0$. Then, for any $c \geq \dfrac{1}{2}$, there is a subset $\sigma \subset [n]$ such that 
\[|\sigma| = cn,\]
and
\[\lambda_{max} (A_{\sigma}) \leq 2\sqrt{c-c^{2}}.\]
\end{theorem}

By considering a diagonal matrix with half the entries $1$ and the other hand $-1$, we see that the condition $c\geq \dfrac{1}{2}$ is necessary. Even if one forces the entire diagonal to be zero, one cannot escape this condition: The $2n \times 2n$ matrix $ \left( \begin{array}{cc}
0 & I\\
I & 0\end{array} \right)$ has zero diagonal but any submatrix of size $n+1$ has norm $1$. One might wonder if adding some other condition might allow us to push beyond this limit. A random GUE matrix(normalized to have norm $1$) has principal submatrices of size $cn$ with norm concentrated strongly around $\sqrt{c}$, after all. I don't know the answer to this : All estimates in this paper are spectral(they depend only on the eigenvalues of the matrices in question). It is natural to add other combinatorial constraints, say on the sizes of the matrix entries to derive refined estimates, but I do not do this in this paper.

By considering $-A$ instead, we see that one gets an analogous for the smallest root as well(thought the set selected could be different). Working with $A_{\sigma}$ now, we can iterate the above arguments to show that by passing to a smaller set, one can control both eigenvalues and hence the norm. The constants in this theorem are however suboptimal, off by a quadratic factor.

\begin{theorem}\label{zd2}
Let $T \in M_n(\mathbb{C})$ be hermitian such that  $\Delta(T) = 0$. Then, for any $c \leq \dfrac{1}{2}$, there is a subset $\sigma \subset [n]$ such that 
\[|\sigma| = c^2n,\]
and
\[||P_{\sigma} T P_{\sigma}|| \leq 2\sqrt{c-c^{2}}.\]
\end{theorem}



By constraining the matrix, one can get better estimates on how small the norm of a principal minor can be. The following theorem includes the case when at most a proportion of roots of a positive contraction are at $1$. In what follows, we use ``$\operatorname{tr}$'' to denote the normalized trace. 
\begin{theorem}\label{zd3}
Let $A \in M_n(\mathbb{C})$ be a positive contraction such that $\operatorname{tr}(A) = \alpha $. Then, for any $c \leq 1-\alpha$, there is a subset $\sigma \subset [n]$ such that 
\[|\sigma| = cn,\]
and
\[||A_{\sigma}|| \leq \left(\sqrt{(1-c)\alpha}+\sqrt{(c(1-\alpha)}\right)^2 .\]
\end{theorem}
Note that the estimate is strictly smaller than $1$ for $c <1- \alpha$. 

The above analysis can be refined to take into account that the positive contraction $A$ in $M_n(\mathbb{C})$ may have plenty of eigenvalues situated away from $1$. A natural way of marking this is by looking at the following quantity, 
\[\dfrac{[\operatorname{tr}(1-A)]^2}{\operatorname{tr}[(1-A)^2]}.\] This is a number between $\operatorname{tr}(1-A)$ and $1$ and values of this quantity away from $1$ indicate that there are several roots away from the end points $0$ and $1$. This expression also appears in follow up work of MSS on restricted invertibility and the utility of these has been discussed by Assaf Naor and Pierre Youssef in their recent paper \cite{NaoYou}. The following theorem shows that we can get principal submatrices with well controlled norm all the way up to this quantity, which can be much larger than the conventional stable rank of the matrix.

\begin{theorem}
For any $0 \leq \delta \leq1$ and any positive contraction $A \in M_n(\mathbb{C})$, there is a principal submatrix $A_S$ of size $cn$ where $c = \delta  \dfrac{\operatorname{tr}(B)^2}{\operatorname{tr}(B^2)}$, where $B = I-A$ such that
\[||A_S|| \leq 1 - \operatorname{tr}(B)\left[\sqrt{1-c}-\sqrt{\delta-c}\right]^2.\]
\end{theorem}

Significantly, this submatrix can be found quickly using a simple algorithm. Applying this theorem to $I-A$ in place of $A$, one can get well invertible principal submatrices.  
\begin{theorem}\label{RI}
For any $0 \leq \delta \leq1$ and any positive contraction $A \in M_n(\mathbb{C})$, there is a principal submatrix $A_S$ of size $cn$ where $c = \delta  \dfrac{\operatorname{tr}(A)^2}{\operatorname{tr}(A^2)}$ such that
\[\lambda_{\operatorname{min}} (A_S) \geq  \operatorname{tr}(A)\left[\sqrt{1-c}-\sqrt{\delta-c}\right]^2.\]
\end{theorem}

The restricted invertibility principle is often stated in terms of linear transformations in the following way: Let $T: \mathbb{C}^m \rightarrow \mathbb{C}^n$ be a linear operator. We are interested in picking a large co-ordinate subset $\sigma \subset [m]$ such that $T\mid_{P_{\sigma}\mathbb{C}^m}$ has all its singular values large.  It is easy to see that the smallest singular value of $T$ is the square root of the smallest eigenvalue of the block compression of $T^{*}T$ given by the elements of $\sigma$. Theorem (\ref{RI}) applied to $T^{*}T$ gives us the following, where use the fact that $\operatorname{Tr}(TT^{*}) = \operatorname{Tr}(T^{*}T)$,
\begin{theorem}
Let $T:  \mathbb{C}^m \rightarrow \mathbb{C}^n$ be a linear operator. Then, for any $0 \leq \delta \leq 1$, there is a subset $\sigma$  of size $|\sigma|  = \delta \dfrac{||T||_2^4}{||T||_4^4}$ and such that, letting $c = \dfrac{|\sigma|}{m}$, we have,
\[s_{min}(T\mid_{P_{\sigma}\mathbb{C}^m}) \geq  \dfrac{||T||_2}{\sqrt{m}}\left[\sqrt{1-c}-\sqrt{\delta-c}\right].\]
\end{theorem}

\begin{remark}\label{MSSIRI}
Assaf Naor and Pierre Youssef in their beautiful paper \cite{NaoYou}, also prove improved restricted invertibility results: Our results are not strictly comparable to theirs and should be seen as complementary. Further, they mention a currently unpublished result of Marcus, Spielman and Srivastava from 2013 \cite{MSS3}, which achieves almost the same estimates as in this last theorem: They are able to get well invertible matrices upto size $\dfrac{1}{4} \dfrac{||T||_2^4}{||T||_4^4}$. The lower bound on the smallest singular value is in our notation, $\dfrac{||T||_2}{\sqrt{m}}\left[\sqrt{1-2\sqrt{\delta}}\right]$, which is slightly weaker than ours. 
\end{remark}

\begin{remark}
Working with principal submatrices yields an alternate combinatorial approach in the spirit of MSS to the Kadison-Singer problem, allowing us to prove Anderson's paving conjecture directly, without using Weaver's $KS_r$ as an intermediate step. In a companion paper, we are able to get useful estimates and also a combinatorial statement that would imply optimal paving estimates.
 \end{remark}
 
 We next prove an extension on the result concerning the roots of higher derivatives of real rooted polynomials to the non real-rooted case. The fundamental Gauss-Lucas theorem\cite{RS02}[2.1] says that the critical points of a univariate polynomial lie in the convex hull of the polynomial's roots. Given a polynomial $p$ and a positive integer $k$, we let $p^{(k)}$ denote the $k$'th derivative of $p$. We also use the notation $\sigma(p)$ to denote the roots of $p$ and $\mathcal{K}(p)$ to denote the convex hull of the roots of $p$. Letting $n$ be the degree of $p$, we have a nested collection of convex sets, 
   \[\mathcal{K}(p) \supset \mathcal{K}(p') \supset \mathcal{K}(p^{(2)}) \supset \cdots \supset \mathcal{K}(p^{(n-1)}).\]
   
   It is easy to see that if let $\alpha$ be the average of the elements in $\sigma(p)$, the average of the elements in $\sigma(p^{(k)})$ equals $\alpha$ as well, for every $1 \leq k \leq n-1$. In particular, the convex sets $\mathcal{K}(p^{(k)})$ shrink to the one element set $\mathcal{K}(p^{(n-1)}) = \{\alpha\}$. It is natural to ask how quickly the sizes of these sets can shrink, something that we could not find a reference to in the literature. We prove the following universal estimate, where given a set $A$  in the plane, $|A|$ refers to the area of $A$. 
   \begin{theorem}\label{convex}
   Let $p$ be a degree $n$ polynomial. Then, for any $c \geq \dfrac{1}{2}$, we have that,
   \[|\mathcal{K}(p^{(\lceil cn \rceil)})| \leq 4(c-c^2)\,|\mathcal{K}(p)|.\]
   \end{theorem}
   
   Note that this estimate $4(c-c^2)$ is independent of the polynomial or even the degree $n$. These estimates are certainly not sharp but we suspect that the $O(1-c)$ dependance is. Also, by looking at the polynomial $p(z) = (z^3-1)^m$, one sees that one needs to take the derivative at least $\frac{n}{3}-2$ times where $n = 3m$ is the degree of $p$, in order to get a shrinking of the areas of the convex hulls of higher derivatives. The theorem (\ref{convex}) as stated above cannot, by this simple observation, hold for $c \leq \frac{1}{3}$. It is conceivable that estimates could be got for $c$ in the range $[\frac{1}{3}, \frac{1}{2}]$, but we do not do this in this paper.
   
   
   The proof is a translation of the results in the real-rooted case to the complex rooted case using the notion of majorization between real sequences by applying results of Pereira \cite{PerKat} and Malamud \cite{Mal05}. This will allow us to prove estimates on root shrinking in each direction. Deducing estimates on the shrinking of the areas of the convex hulls will then be a simple corollary.

\section*{Acknowledgement}
I'd like to specially thank Betul Tanbay who started me off on this paper by asking if perhaps the machinery of MSS could be directly applied to Anderson's paving conjecture and who was very generous with her time. Another special thanks to Ozgur Martin, for all his help and insights. I'd also like to thank Amit Deshpande and Atilla Yilmaz for useful discussions. 

\section{The method of Interlacing polynomials}

Let $A$ be a hermitian matrix in $M_n(\mathbb{C})$ and let $A_k$ for $1 \leq k \leq n$ be the principal submatrices constructed by removing the $k'th$ row and column from $A$. The celebrated interlacing theorem of Cauchy says that the eigenvalues of $A_k$ interlace those of $A$. Writing this out in terms of characteristic polynomials, we have that the polynomials $\chi[A_k]$ all interlace the polynomial $\chi[A]$. As pointed by MSS, this implies that there is a some $k$ such that 
\begin{eqnarray}\label{MSSI}
\operatorname{max root} \chi[A_k] \leq \operatorname{max root} \sum \chi[A_k].
\end{eqnarray}
The last sum is well known, due to a theorem of R.C.Thompson \cite{ThoPS1}, 
\begin{theorem}[R.C.Thompson]\label{RCT}
Let $A \in \mathbb{C}$ and let $A_k$ be its defect $1$ principal submatrices. Then, 
\[\sum \chi[A_k] = \chi'[A].\]
\end{theorem}
Combining (\ref{MSSI}) and theorem (\ref{RCT}), we conclude that given a hermitian $A \in M_n(\mathbb{C})$, there is is a $k$ such that 
\[\operatorname{max root} \chi[A_k] \leq \operatorname{max root}\chi'[A].\]

We can make a similar statement about the largest roots of repeated derivatives of the principal submatrices. 

\begin{lemma}\label{ind}
Let $A \in M_n(\mathbb{C})^{sa}$. Then, for any $m$, there is a $k$ such that 
\[\operatorname{max root} \chi^{(m)}[A_k] \leq \operatorname{max root} \chi^{(m+1)}[A].\]
\end{lemma}
\begin{proof}
As noted above, the polynomials $\chi[A_k]$ for $k \in [n]$ have a common interlacer, namely $\chi[A]$. As pointed out by MSS, the property of having common interlacers is preserved under taking derivatives. This means that the polynomials, $\chi^{(m)}[A_k]$ for $k \in [n]$ have a common interlacer. Applying MSS's Markov principle for this set, we conclude that there is a $k$ such that 
\[\operatorname{max root} \chi^{(m)}[A_k] \leq \operatorname{max root} \sum \chi^{(m)}[A_k] = \operatorname{max root} \chi^{(m+1)}[A].\]
\end{proof}

We may now iterate this to get bounds for principal submatrices of any size. 
\begin{theorem}\label{inter high}
Let $A \in M_n(\mathbb{C})^{sa}$. Then, for any $m$, there is a principal submatrix $A_S$ of size $m$ such that 
\[\operatorname{max root} \chi[A_S] \leq \operatorname{max root}\chi^{(n-m)}[A].\]
\end{theorem}
\begin{proof}
Using lemma (\ref{ind}), we may find find a principal submatrix $A_{S_1}$ of $A$ of size $n-1$ such that 
\[\operatorname{max root} \chi^{(n-m-1)}[A_{S_1}] \leq \operatorname{max root}\chi^{(n-m)}[A].\]
We next find a principal submatrix, $A_{S_2}$ of $A_{S_1}$ of size $n-2$ such that 
\[\operatorname{max root} \chi^{(n-m-2)}[A_{S_2}] \leq \operatorname{max root}\chi^{(n-m-1)}[A^1].\]
Iterating this a total of $n-m$ times, we get a principal submatrix $A_S :=A_{S_{n-m}}$ of size $m$ such that 
\[\operatorname{max root} \chi[A_S] \leq \operatorname{max root}\chi^{(n-m)}[A].\]
\end{proof}

\begin{remark}
This above process is algorithmic: One starts off with $A$ and compares the largest roots of $\chi^{(n-m-1)}(A_k)$ for $k \in [n]$ and selects the one such that the largest root is minimal. We then look at \emph{its} defect $1$ principal submatrices and select the one with the minimal largest root for $\chi^{(n-m-2)}$. We iterate this process a total of $n-m$ times to get the desired size $m$ principal submatrix.
\end{remark}

\section{The Batson-Marcus-Spielman-Srivastava barrier method}
The barrier method, introduced by Batson, Spielman and Srivastava \cite{BSS} and further clarified by Spielman and Srivastava \cite{SSEle} and Marcus, Spielman and Srivastava \cite{MSS2}, is a general method for getting estimates for the largest root of a real rooted polynomial. 

Given a real rooted polynomial $p$, Spielman et. al. define the potential function of $p$ by 
\[\Phi_p(b) := \dfrac{p'(b)}{p(b)} = \sum \dfrac{1}{b-\lambda_i}, \quad b > \lambda_{\operatorname{max}}.\]

This is a positive, monotone decreasing, convex function. They then use the inverse of this function to define a quantity called  $\operatorname{smax}$, a ``soft maximum'' for the largest root,
\[smax_{\varphi}(p) = b, \quad \text{ if } \Phi_p(b) = \varphi.\]
Here $\varphi$ is any positive real number and as MSS point out, $smax_{\varphi}(p) $ gives an upper bound for the largest root, with the precision of the bound controlled by the ``sensitivity'' parameter $\varphi$ (the precision increasing as $\varphi$ does). The utility of this function comes from the fact that this behaves in a controlled fashion when applying linear differential operators to the polynomial $p$. MSS apply this technique to control the largest roots of $(1-\dfrac{d}{dx})^m p$. When $p = x^n$, this is enough to yield the restricted invertibility principle in the isotropic case.

We will apply this method directly to the derivative operator.  We have that 
\begin{eqnarray}\label{der}
\Phi_{p'} = \dfrac{(p\Phi_p)'}{p\Phi_p} = \Phi_p+ \dfrac{\Phi'_p}{\Phi_p},
\end{eqnarray}
which shows that $\Phi_{p'}$ is smaller than $\Phi_p$. And since, $\Phi$ is decreasing, we have that $\operatorname{smax}_{\varphi}(p') \leq \operatorname{smax}_{\varphi}(p)$ for any $\varphi$. The key to the barrier  method is the following more refined estimate,
\begin{proposition}\label{barmove}
Let $p$ be a real rooted polynomial and $\varphi \in (0,\infty]$. Then, 
\[\operatorname{smax}_{\varphi}(p') \leq \operatorname{smax}_{\varphi}(p) - \dfrac{1}{\varphi}.\]
\end{proposition}

\begin{proof}
It is easy to see that $\dfrac{1}{\Phi_p(b)}$ is positive, increasing and concave, yielding that for any $b > \lambda_{max}(p)$ and $\delta  > 0$ such that $b - \delta > \lambda_{max}(p)$, 
\[\dfrac{1}{\Phi_p(b-\delta)} - \dfrac{1}{\Phi_p(b)} \leq \delta\left(\dfrac{1}{\Phi_p(b-\delta)}\right)'.\]
Let $\varphi = \Phi_p(b)$. Note that, 
\[\dfrac{1}{b - \lambda_{max}(p) } < \sum \dfrac{1}{b-\lambda_i(p) }= \Phi_p(b) = \varphi, \]
yielding that  $b - \dfrac{1}{\varphi} > \lambda_{max}(p)$. We now have that,
\[\Phi_p(b) - \Phi_p(b-\dfrac{1}{\varphi}) \leq \dfrac{\Phi'_p(b-\dfrac{1}{\varphi})}{\Phi_p(b-\dfrac{1}{\varphi})}.\]
By (\ref{der}), we have that 
\[\Phi_{p'}(b-\dfrac{1}{\varphi}) = \Phi_p(b-\dfrac{1}{\varphi})+ \dfrac{\Phi'_p(b-\dfrac{1}{\varphi})}{\Phi_p(b-\dfrac{1}{\varphi})},\]
yielding that 
\[\Phi_{p'}(b-\dfrac{1}{\varphi}) \leq \Phi_p(b), \quad \text{ where } \varphi = \Phi_p(b).\]

Since $\Phi$ is decreasing, we conclude that 
\[\operatorname{smax}_{\varphi}(p') \leq \operatorname{smax}_{\varphi}(p) - \dfrac{1}{\varphi}.\]

\end{proof}

Iterating this and noting that $\lambda_{max}(p) = \operatorname{inf}_{\varphi \geq 0}\, \operatorname{smax}_{\phi}(p)$ for any real rooted polynomial we see that 
\begin{proposition}\label{bar}
Let $p$ be a real rooted polynomial and $\varphi \in (0,\infty]$. Then, for any $k \leq \operatorname{deg}(p)$, we have,
\[\lambda_{max}(p^{(k)}) = \operatorname{inf}_{\varphi \geq 0} \,\, \operatorname{smax}_{\varphi}(p^{(k)}) \leq   \operatorname{inf}_{\varphi \geq 0} \,\ \operatorname{smax}_{\varphi}(p) - \dfrac{k}{\varphi}.\]
\end{proposition}

We have so far proceeded by noting that there is one submatrix $A_S$ of size $n-k$ whose largest eigenvalue can be controlled by the largest root of the $k$'th derivative of $\chi[A]$. This submatrix can be found iteratively, as pointed out in the last paragraph of section $(2)$. There is another pleasant algorithm to find this submatrix that is perhaps even more direct. 

The interlacing property of submatrices, which allows us to get information on eigenvalues can also be used analogously for quantities $\operatorname{smax}_{\phi}(p)$. The largest eigenvalue is the special case when $\phi = \infty$ and the same fact, that there is one principal defect $1$ submatrix whose largest eigenvalue is at most the largest root of the derivative also holds for these other quantities.

\begin{proposition}
Let $A \in M_n(\mathbb{C})^{sa}$ and let $\phi \in [0,\infty]$. Then, there is a defect $1$ principal submatrix $A_i$ such that,
\[\operatorname{smax}_{\varphi}(A_i) \leq \operatorname{smax}_{\varphi}(\chi'[A]).\]
\end{proposition} 
\begin{proof}
It is a well known fact that,
\[\chi[A_i](x) = \chi[A] (x) e_i^*(xI-A)^{-1}e_i.\]
Let $(\lambda_1, \cdots, \lambda_n)$ be the eigenvalues of $A$ and let $D$ be the diagonal matrix, $D = \operatorname{diag}(\lambda_1, \cdots, \lambda_n)$. Choose a unitary $U$ such that $A = U^{*}DU$. We now see that we may write, 
\[\chi[A_i](x) = \chi[A](x)\sum_{j \in [n]} \dfrac{|U_{ij}|^2}{x - \lambda_j}.\]
We now see that,
\begin{eqnarray*}
\Phi_{\chi[A_i]}(x) = \dfrac{\chi'[A_i]}{\chi[A_i]} (x) &=& \dfrac{\chi'[A]}{\chi[A]}(x) - \dfrac{\sum_{j \in [n]}\dfrac{|U_{ij}|^2}{(x - \lambda_j)^2}}{\sum_{j \in [n]}\dfrac{|U_{ij}|^2}{x - \lambda_j}},\\
&=&\Phi_{A}(x) - \dfrac{\sum_{j \in [n]}\dfrac{|U_{ij}|^2}{(x - \lambda_j)^2}}{\sum_{j \in [n]}\dfrac{|U_{ij}|^2}{x - \lambda_j}}\\
&=& \Phi_{\chi'[A]}(x) +  \dfrac{\sum_{j \in [n]}\dfrac{1}{(x - \lambda_j)^2}}{\sum_{j \in [n]}\dfrac{1}{x - \lambda_j}}
 - \dfrac{\sum_{j \in [n]}\dfrac{|U_{ij}|^2}{(x - \lambda_j)^2}}{\sum_{j \in [n]}\dfrac{|U_{ij}|^2}{x - \lambda_j}}
\end{eqnarray*}
Suppose now that $\Phi_{\chi[A_i]}(x) > \Phi_{\chi'[A]}(x)$ for all $i \in [n]$. We then have that, for all $i \in [n]$,
 \[  \dfrac{\sum_{j \in [n]}\dfrac{1}{(x - \lambda_j)^2}}{\sum_{j \in [n]}\dfrac{1}{x - \lambda_j}}
 > \dfrac{\sum_{j \in [n]}\dfrac{|U_{ij}|^2}{(x - \lambda_j)^2}}{\sum_{j \in [n]}\dfrac{|U_{ij}|^2}{x - \lambda_j}}., \]
 which we may write as, 
 \[\left(\sum_{j \in [n]}\dfrac{1}{(x - \lambda_j)^2}\right) \,\left(\sum_{j \in [n]}\dfrac{|U_{ij}|^2}{x - \lambda_j}\right)> \left(\sum_{j \in [n]} \dfrac{1}{x - \lambda_j} \right)\,\left(\sum_{j \in [n]}\dfrac{|U_{ij}|^2}{(x - \lambda_j)^2}\right), \quad i \in [n] .\]
The matrix $U$ is unitary and thus, the columns are unit vectors, which means that $\sum_{j \in [n]}|U_{ij}|^2 = 1$ for each $i \in [n]$. Summing the above expression over $i$, we see that, 
 \[\left(\sum_{j \in [n]}\dfrac{1}{(x - \lambda_j)^2}\right) \,\left(\sum_{j \in [n]} \dfrac{1}{x - \lambda_j} \right)> \left(\sum_{j \in [n]} \dfrac{1}{x - \lambda_j} \right)\,\left(\sum_{j \in [n]}\dfrac{1}{(x - \lambda_j)^2}\right),\]
 a contradiction. We conclude that for every $x$, there is a $i$ such that,
 \[\Phi_{\chi[A_i]}(x)  \leq  \Phi_{\chi'[A]}(x) \]
 For any $\varphi$, let $x$ be such that $\Phi_{\chi'[A]}(x) = \varphi$. There is a $i$ such that $\Phi_{\chi[A_i]}  \leq \varphi$. Since $\Phi_{\chi[A_i]}$ is decreasing, we conclude that,
 \[\operatorname{smax}_{\varphi}(A_i) \leq \operatorname{smax}_{\varphi}(\chi'[A]).\]

\end{proof}

Together with (\ref{barmove}), we conclude, 
\begin{proposition}
Let $A \in M_n(\mathbb{C})^{sa}$ and let $\phi \in [0,\infty]$. Then, there is a defect $1$ principal submatrix $A_i$ such that,
 \[\operatorname{smax}_{\varphi}(A_i) \leq \operatorname{smax}_{\varphi}(\chi[A]) - \dfrac{1}{\varphi}.\]
\end{proposition}

\begin{remark}
This immediately gives us a sublime algorithm for getting principal submatrices with small largest eigenvalue. Fix a potential $\varphi$, and sequentially find defect $1$ submatrices with minimum $\operatorname{smax}_{\varphi}$. If $\varphi$ is chosen properly, see the next section, this will give us optimally small submatrices.
 \end{remark}

\section{Optimization}
With an estimate for the largest root of the $k$'th derivative in hand, let us now optimize this under various hypotheses. From now on, without loss of generality, we will work with positive contractions or equivalently, real rooted polynomials all of whose roots are in $[0,1]$. Our first hypothesis is the most natural way of ensuring that not all the roots are $1$, that is, avoiding the case when $A$ is the identity matrix. We simply demand that the average of the roots is some number $\alpha$ which will be taken to be less than $1$. We will use the following elementary lemma. 

\begin{lemma}\label{HM}
 
Suppose we are given a $b > 1$ and a collection of numbers $\lambda_1, \cdots, \lambda_n$ in $[0,1]$ with average $\alpha$, Then, 
\[ \sum \dfrac{1}{b - \lambda_i} \leq \dfrac{\alpha n}{b-1} + \dfrac{(1-\alpha)n}{b}\]
\end{lemma}

We give a proof of this elementary fact in order to keep this paper self-contained. 

\begin{proof}
  Suppose two,  of the numbers are inside $(0,1)$, so that without loss of generality, we have $0 <  a_1 \leq a_2  < 1$. It is easy to see that
 \[\dfrac{1}{a_1 - \epsilon} + \dfrac{1}{a_2+\epsilon} > \dfrac{1}{a_1} + \dfrac{1}{a_2}\] for sufficiently small values of $\epsilon$. Replacing the set $(a_1, a_2, \cdots)$ by $(a_1-\epsilon, a_2+\epsilon, \cdots)$ increases the value of the expression $\varphi = \sum \dfrac{1}{b-\lambda_i}$. It is now easy to see that the quantity is maximised when all the numbers, save at most one are either $0$ or $1$. Let $k = \lfloor n\alpha \rfloor$ and let $x = n\alpha - k$, which is a number in $[0,1]$. We than have that for every collection of numbers $\{\lambda_i\}$ satisfying the hypotheses, 
 \[\varphi = \sum \dfrac{1}{b-\lambda_i}\leq \dfrac{k}{b-1} + \dfrac{n-k-1}{b} + \dfrac{1}{b-x}.\]
 We also have by the harmonic mean inequality that,
 \[\dfrac{1}{b-x} \leq \dfrac{x}{b-1}+\dfrac{1-x}{b}.\]
 Adding these two inequalities, we get the desired result. 
\end{proof}

\begin{theorem}\label{mrr}
Let $p$ be a real rooted polynomial of degree $n$ with roots lying in $[0,1]$ and with the average of the roots $\alpha$. Then, for any $c \geq \alpha$, 
\[\operatorname{max \,root}(p^{(cn)}) \leq  \left(\sqrt{(1-\alpha)(1-c)}+\sqrt{\alpha\,c}\right)^2.\]
\end{theorem}

\begin{proof}
By proposition (\ref{bar}), for any $b$, the quantity
\[b - \dfrac{cn}{\varphi(b)},\qquad \varphi(b) = \sum \dfrac{1}{b - \lambda_i},\]
is an upper bound for the largest root of $p^{(cn)}$. By lemma (\ref{HM}), we have that

\[\varphi(b) \leq \dfrac{\alpha n}{b-1} + \dfrac{(1-\alpha)n}{b} = \dfrac{n[b-(1-\alpha)]}{b(b-1)}.\]
This in turn, using theorem (\ref{bar}) yields that for every $b>1$, the following quantity upper bounds the maximum root,  
\begin{eqnarray}\label{cal}b - \dfrac{cn}{\varphi} = b - \dfrac{cb(b-1)}{b-(1-\alpha)} = b(1-c) + c\alpha + \dfrac{c\alpha(1-\alpha)}{b-(1-\alpha)}.
\end{eqnarray}
This expression, as a function of $b$ equals $1$ when $b = 1$, goes to $\infty$ as $b$ goes to infinity and is unimodal, decreasing to a unique global minimum and increasing subsequently. We calculate the critical point, getting that
\[b  = (1-\alpha) + \sqrt{\dfrac{c}{1-c}}\sqrt{\alpha(1-\alpha)}.\]
Substituting this in (\ref{cal}), we see that the largest root of $p^{(cn)}$ is bounded by 
\[(\sqrt{(1-\alpha)(1-c)}+\sqrt{c\alpha})^2.\]
This expression is strictly less than $1$ for any $c > \alpha$.
\end{proof}

We will use this simple optimization result again in what follows and we record it.
\begin{lemma}\label{alpha}
For any $x \in [0,1]$, the expression 
 \[b -  \dfrac{cn}{\varphi }, \quad \text{where } \varphi = \dfrac{\alpha n}{b-1} + \dfrac{(1-\alpha)n}{b-x}\]
 for $b > 1$ has minimum value equal to 
 \[x + (1-x)(\sqrt{(1-\alpha)(1-c)}+\sqrt{c\alpha})^2.\]
\end{lemma}
\begin{proof}
 When $x = 0$, this follows from the calculation in the previous proof. In general, this follows by making the substitution $\tilde{b} = \dfrac{b-x}{1-x}$. 
\end{proof}

The case when the average of the roots $\alpha$ is $\dfrac{1}{2}$, after translation and scaling, yields the following remarkable fact, for which I could not find a reference in the literature. 
\begin{theorem}\label{rzerosh}
Let $p$ be a real rooted polynomial of degree $n$ with roots lying in $[-1,1]$ and summing upto $0$. Then, for any $c \geq \dfrac{1}{2}$, 
\[\operatorname{roots}(p^{(cn)}) \subset [-2\sqrt{c-c^2},2\sqrt{c-c^2}].\]
\end{theorem}

We now refine this analysis to take into account that the roots might be spread out, rather than concentrated at the end points $0$ and $1$. Let us demand that apart from the roots lying in $[0,1]$, we also have that
\[\sum \lambda_i = n \alpha, \qquad \sum \lambda_i^2 = n \beta.\]
It is immediate that $\alpha^2 \leq \beta \leq \alpha$, the first by Cauchy-Schwarz and the second by the condition that the roots lie in $[0,1]$. Under these constraints, we would like to see when the potential is maximized. 
\begin{lemma}
\label{lemma:associativity}
Suppose $\lambda_i$ for a $i \in [n]$ are a collection of real numbers in $[0,1]$ satisfying 
\[\sum \lambda_i = n \alpha, \qquad \sum \lambda_i^2 = n \beta.\]
Then, for any fixed $b > 1$, the quantity 
\[\varphi = \sum \dfrac{1}{b - \lambda_i},\] satisfies,
\[\varphi \leq \dfrac{ns}{b-1} + \dfrac{nt}{b-x}\]
where $s, t, x$ are given by, 
\[x = \dfrac{\alpha - \beta}{1-\alpha}, \qquad s = \dfrac{\beta - \alpha^2}{1-2\alpha + \beta}, \qquad t = 1- s = \dfrac{(1-\alpha)^2}{1-2\alpha+\beta}.\]
\end{lemma}

We relegate the proof of this fact, which is elementary, but tedious, to the appendix. This lemma allows us to prove a strong restricted invertibility result, which shows that one can get well conditioned principal submatrices of size right up to the modified stable rank.

\begin{theorem}\label{KasTza}
Let $A \in M_n(\mathbb{C})$ be a positive contraction and let $B = I-A$. Then, for any $c \leq \dfrac{\operatorname{tr}(B)^2}{\operatorname{tr}(B^2)}$, written as $c = \delta \dfrac{\operatorname{tr}(B)^2}{\operatorname{tr}(B^2)}$ for some $0 \leq \delta \leq 1$, there is a principal submatrix $A_S$ of size $cn$ such that 
\[||A_S|| \leq 1- \operatorname{tr}(B)\left(\sqrt{1-c}-\sqrt{\delta- c}\right)^2\]
\end{theorem}
\begin{proof}
Combining proposition (\ref{bar}) and lemmas (\ref{alpha}) and (\ref{xest}), we see that there is a principal  submatrix $A_S$ of size $cn$ which satisfies, 
\begin{eqnarray}\label{com2}
\lambda_{max}(A_S) \leq x + (1-x)\left(\sqrt{(1-\alpha)c}+\sqrt{(1-c)\alpha}\right)^2,
\end{eqnarray}
where, 
\[x = \dfrac{\operatorname{tr}(A) - \operatorname{tr}(A^2)}{1-\operatorname{tr}(A)} = \dfrac{\operatorname{tr}(B) - \operatorname{tr}(B^2)}{\operatorname{tr}(B)},\]
and, 
\[\alpha = \dfrac{\operatorname{tr}(A^2) - \operatorname{tr}(A)^2}{1-2\operatorname{tr}(A)+\operatorname{tr}(A^2)}= \dfrac{\operatorname{tr}(B^2) - \operatorname{tr}(B)^2}{\operatorname{tr}(B^2)}.\] 
Also, note that this is strictly less than $1$ for $1-c > \alpha$ which reduces to $c < \dfrac{\operatorname{tr}(B)^2}{\operatorname{tr}(B^2)}$. We have that $c = \delta \dfrac{\operatorname{tr}(B)^2}{\operatorname{tr}(B^2)}$, which allows us to write, 
\[x = 1-\dfrac{\delta \operatorname{tr}(B)}{c}, \quad \alpha = \dfrac{\delta - c}{\delta}.\]
The expression in (\ref{com2}) simplifies to, 
\begin{eqnarray*}
&&1-\dfrac{\delta \operatorname{tr}(B)}{c} + \dfrac{\delta \operatorname{tr}(B)}{c} \left[\sqrt{\dfrac{c^2}{\delta}} + \sqrt{\dfrac{(1-c)(\delta-c)}{\delta}}\right]^2\\
&=& 1-\dfrac{\delta \operatorname{tr}(B)}{c} + \dfrac{\operatorname{tr}(B)}{c} \left[c^2 + (1-c)(\delta-c)+2c\sqrt{(1-c)(\delta-c)}\right]\\
&=& 1+\operatorname{tr}(B)\left[2c - 1 - \delta + 2\sqrt{(1-c)(\delta-c)}\right]\\
&=& 1- \operatorname{tr}(B)\left(\sqrt{1-c}-\sqrt{\delta- c}\right)^2
\end{eqnarray*}
\end{proof}

Working with $1-A$ in place of $A$ and using that $\lambda_{min}(A) = 1 - \lambda_{max}(1-A)$, we have the restricted invertibility principle, 
\begin{theorem}\label{BT}
Let $A \in M_n(\mathbb{C})$ be a positive contraction. Then, for any $c \leq \dfrac{\operatorname{tr}(A)^2}{\operatorname{tr}(A^2)}$, written as $c = \delta \dfrac{\operatorname{tr}(A)^2}{\operatorname{tr}(A^2)}$ for some $0 \leq \delta \leq 1$, there is a principal submatrix $A_S$ of size $cn$, which equals $\dfrac{\operatorname{Tr}(A)^2}{\operatorname{Tr}(A^2)}$, such that 
\[\lambda_{min}A_S \geq  \operatorname{tr}(A)\left(\sqrt{1-c}-\sqrt{\delta- c}\right)^2\]
\end{theorem}

Following the simple argument in the last paragraph of the introduction, we have the following version of the restricted invertibility theorem, 
\begin{theorem}
Let $T:  \mathbb{C}^m \rightarrow \mathbb{C}^n$ be a linear operator. Then, for any $0 \leq \delta \leq 1$, there is a subset $\sigma$  of size $|\sigma|  = \delta \dfrac{||T||_2^4}{||T||_4^4}$ and such that, letting $c = \dfrac{|\sigma|}{m}$, we have,
\[s_{min}(T\mid_{ P_{\sigma}\mathbb{C}^m}) \geq  \dfrac{||T||_2}{\sqrt{m}}\left[\sqrt{1-c}-\sqrt{\delta-c}\right].\]
\end{theorem}

   \section{Majorization relations for polynomial roots}

 We start off with a result proved by Pereira in 2005 \cite{PerKat} and conjectured by Katsoprinakis in the 1980's \cite{Kat}. The result also appears in the contemporaneous work of Malamud \cite{Mal05} on closely related problems. Recall that a real sequence $\overline{\mu}$ is majorized by a real sequence $\overline{\lambda}$(of the same size), which we will denote $\overline{\mu} \prec \overline{\lambda}$ if there is a doubly stochastic map $D$ such that $D\overline{\lambda} = \overline{\mu}$. Here, a doubly stoachastic map is a matrix of non-negative reals with all row and column sums $1$. It is a classical fact that Majorization can also be expressed in terms of convex maps, in the following way, see \cite{PerKat}[Prop. 4.2], 
   \begin{theorem}\label{conv}
   Let $\overline{\mu} = (\mu_1, \cdots, \mu_n)$ and $ \overline{\lambda} = (\lambda_1, \cdots, \lambda_n)$ be two real sequences. Then, the following are equivalent, 
   \begin{enumerate}
   \item $\overline{\mu} \prec \overline{\lambda}$
   \item For every convex function $f$ defined on an interval containing both  $\overline{\lambda} $ and $\overline{\mu}$, we have that,
   \[\sum f(\mu_i) \leq \sum f(\lambda_i).\]
   \end{enumerate}
   \end{theorem}
   
   Given a polynomial $p$ with roots $(\lambda_1, \cdots, \lambda_n)$, we will use the notation  $R(p)$ to denote the monic polynomial whose roots are  $(\operatorname{Re}\,\lambda_1, \cdots, \operatorname{Re}\, \lambda_n)$. The following was conjectured by Katsoprinakis\cite{Kat} and proved 20 years later by Pereira \cite{PerKat}[Theorem 4.6] and independently by Malamud \cite{Mal05},
   \begin{theorem}[Pereira]\label{Per}
  Given a polynomial $p$, we have, 
   \[\sigma \left(R(p')\right) \prec \sigma \left(R(p)' \right).\]
   \end{theorem}

We would like to point another interesting relation between roots of real parts of polynomials and their derivatives. It will be convenient to use the notation $Dp$ to represent the derivative of $p$. The theorems of Pereira and Malamud show that, 
\begin{eqnarray}\label{PM}
\sigma(RD(p)) \prec \sigma(DR(p)).
\end{eqnarray}
We now show that there is an interesting extension for higher derivatives. 
\begin{theorem}
Let $p$ be a degree $n$ polynomial and let $k \leq n$. We then have a chain of majorization relations (between real sequences of size $n-k$), 
\[\sigma\left(RD^{(k)}(p)\right) \prec \sigma\left(DRD^{(k-1)}(p)\right) \prec \cdots \prec \sigma\left(D^{(k-1)}R(p)\right) \prec \sigma\left(D^{(k)}R(p)\right).\]
\end{theorem}
\begin{proof}
Applying (\ref{PM}) to the polynomial $D^{(k-1)}(p)$ yields that, 
\[\sigma\left(RDD^{(k-1)}(p)\right) = \sigma\left(RD^{(k)}(p)\right) \prec \sigma\left(DRD^{(k-1)}(p)\right).\]

Borcea and Branden in \cite{BB10}[Theorem 1] showed(this is a very special case of their theorem) that if $q$ and $p$ are real rooted polynomials, then $\sigma(q) \prec \sigma(p)$ implies that $\sigma(Dq) \prec \sigma(Dp)$. Applying this to the polynomials $RD^{(k-1)}(p)$ and  $DRD^{(k-2)}(p)$, we see, using (\ref{PM}) again that, 
\[ \sigma\left(DRD^{(k-1)}(p)\right) \prec \sigma\left(DDRD^{(k-2)}(p)\right) = \sigma\left(D^{(2)}RD^{(k-2)}(p)\right) .\]
Iterating this argument establishes the theorem. 
\end{proof}

We will only need one consequence of this theorem. 
\begin{corollary}\label{majrc}
Let $p$ be a degree $n$ polynomial and let $k \leq n$. Then, 
\[\lambda_{max} \left(RD^{(k)}(p)\right) \leq \lambda_{max} \left(D^{(k)}(R(p))\right), \quad \lambda_{min} \left(RD^{(k)}(p)\right) \geq \lambda_{min} \left(D^{(k)}(R(p))\right).\]
\end{corollary}

\section{A quantitative Gauss-Lucas theorem}
\begin{lemma}\label{rrs}
   Let $p$ be a real rooted polynomial of degree $n$. Then, for any $c \geq \frac{1}{2}$, we have, letting $|\sigma(p)| = \operatorname{\lambda}_{max}(p) - \operatorname{\lambda}_{min}(p)$, that,
\[|\sigma(p^{(cn)})| \leq 2\sqrt{c-c^2}\, |\sigma(p)|.\] 
   \end{lemma}
   \begin{proof}
It is easy to see that shifting and scaling the roots of the polynomial $p$ does not affect the ratio $\dfrac{|\sigma(p^{(cn)})|}{ |\sigma(p)|}$ and  we may therefore assume that the polynomial $p$ has roots in $[-1,1]$. Let $\alpha$ be the average of the roots of $p$. Applying theorem(\ref{mrr}) to the polynomial $q(z) = p(2z-1)$, which has roots in $[0,1]$ and the average of whose roots is $\frac{1+\alpha}{2}$ , we see that, 
   \[\lambda_{max}(p^{(cn)})) \leq \left(\sqrt{(1-\alpha)(1-c)}+\sqrt{(1+\alpha)\,c}\right)^2 - 1, \quad \text{if } c \geq \dfrac{1+\alpha}{2}.\]
   Working with the polynomial $q(z) = p(1-2z)$, we have that $q$ has roots in $[0,1]$ and the average of its roots is $\frac{1-\alpha}{2}$, and we see that, 
    \[\lambda_{min}(p^{(cn)} ) \geq  1- \left(\sqrt{(1+\alpha)(1-c)}+\sqrt{(1-\alpha)\,c}\right)^2 , \quad \text{if } c \geq \dfrac{1-\alpha}{2}.\]
    Without loss of generality, we may assume that $\alpha \leq  0$ (else we work with $r(z) = p(-z)$ instead. We therefore have that,
     \[ |\sigma(p^{(cn)})| \leq \begin{cases}
  4\sqrt{c(1-c)(1-\alpha^2)}, & c \geq \dfrac{1-\alpha}{2}\\
    \left(\sqrt{(1-\alpha)(1-c)}+\sqrt{(1+\alpha)\,c}\right)^2, &  \dfrac{1+\alpha}{2} \leq c \leq \dfrac{1-\alpha}{2}
    \end{cases}\]
    
    In the case when $c \geq  \dfrac{1-\alpha}{2}$, we note that the expression $4\sqrt{c(1-c)(1-\alpha^2)}$ is maximized when $\alpha = 0$ where it equals $4\sqrt{c(1-c)}$.

     For the case when $ \dfrac{1+\alpha}{2} \leq c \leq \dfrac{1-\alpha}{2}
$, but still, $c \geq \dfrac{1}{2}$: We see that for fixed $c$, the expression, 
\begin{eqnarray}\label{exp}
  \left(\sqrt{(1-\alpha)(1-c)}+\sqrt{(1+\alpha)\,c}\right)^2,
  \end{eqnarray}
  
 as a function of $\alpha$ increases from $-1$  to $2c-1$ and then decreases from $2c-1$ to $1$. We have the condition  $ \dfrac{1+\alpha}{2} \leq c \leq \dfrac{1-\alpha}{2}$  which gives us that $\alpha \leq \operatorname{min}\{2c-1, 1-2c\}$. Together with the condition $c \geq \frac{1}{2}$, this reduces to the condition $\alpha \leq 1-2c$. The expression (\ref{exp}) subject to this constraint on $\alpha$ thus has a maximum value at $\alpha = 1-2c$, where it equals $8c(1-c)$. It is easy to see that this is smaller than $4\sqrt{c(1-c)}$ for every $c \in [0,1]$. And finally, using the fact $p$ has roots in $[-1,1]$, 
\[\dfrac{|\sigma(p^{(cn)})|}{ |\sigma(p)|} \leq \dfrac{4\sqrt{c(1-c)}}{2} = 2\sqrt{c(1-c)}.\]
     \end{proof}
     
     We note a simple corollary,
     \begin{corollary}\label{comsh}
       Let $p$ be a polynomial of degree $n$. Then, for any $c \geq \frac{1}{2}$, we have,
\[|\operatorname{Re}(\sigma(p^{(cn)}))| \leq 2\sqrt{c-c^2}\, |\operatorname{Re}(\sigma(p))|.\] 
     \end{corollary}
     \begin{proof}
     Combine corollary(\ref{majrc}) and lemma(\ref{rrs}).
     \end{proof}
  
  We now deduce our main result, a quantitative Gauss-Lucas theorem, 
  
 \begin{theorem}\label{crs}
   Let $p$ be a polynomial of degree $n$. Then, for any $c \geq \frac{1}{2}$, we have that,
\[|\sigma(p^{(cn)})| \leq 4(c-c^2)\, |\sigma(p)|.\] 
   \end{theorem}
   \begin{proof}
   Corollary (\ref{comsh}) says that the ratio between the sizes of the projections of $\sigma(p^{(cn)})$ and $\sigma(p)$ onto the real axis is at most $2\sqrt{c-c^2}$. There is nothing special about the real axis; Working with $q(z) = p(e^{-i\theta}z)$, we see that the ratios of the projections onto the line $\operatorname{Arg}(z) = \theta$ are again bounded by $2\sqrt{c-c^2}$. We therefore have two polygons with the properties, 
   \begin{enumerate}
   \item The ratios of their shadows in every direction are at most $2\sqrt{c-c^2}$.
   \item They have the same centroid(since the roots of a polynomial and its critical points have the same average). 
   \end{enumerate} 
   Writing out the areas in polar coordinates shows that the ratio of the areas is at most $4(c-c^2)$. 
      \end{proof}
      
      Let us mention another result along these lines. 
      \begin{theorem}
Let $p$ be a degree $n$ polynomial with roots in $B(0,1)$ and with average of its roots $0$. Then, for any $c \geq \frac{1}{2}$,
\[\sigma(p^{(cn)}) \subset B(0, 2\sqrt{c-c^2}). \] 
      \end{theorem}
      \begin{proof}
      The real rooted polynomial $R(p)$ has roots in $(-1, 1)$ and the average of its roots is $0$. Theorem(\ref{rzerosh}) then implies that, 
      \[\sigma\left(R(p^{(cn)})\right) \subset (-2\sqrt{c-c^2}, 2\sqrt{c-c^2}). \] 
      And clearly, the same holds for any other line that we project the roots to. The theorem follows. 
   \end{proof}

\section{Conclusion}
We end with two comments, one on the relationship between the approach in this paper and that of MSS\cite{MSS3} and the second concerning tightness of bounds. 
\subsection{Sampling with and without replacement}

Applying the method of interlacing polynomials to principal submatrices is closely related to the original argument of Marcus, Spielman and Srivastava \cite{MSSICM}. Let us illustrate this with an example. A special case of the restricted invertibility principle is the following,
\begin{theorem}[MSS]\label{RISimple}
Given vectors $v_1, \cdots, v_m \in \mathbb{C}^n$ such that 
\[\sum_{i \in [m]} v_i v_i^{*}  = I,\]
for any $k \leq n$, there is a subset $\sigma$ of size $k$ such that,
\[\lambda_k \left[ \sum_{i \in \sigma} v_i v_i^{*} \right] \geq \left(1 - \sqrt{\dfrac{k}{n}}\right)^2 \dfrac{n}{m}.\]
\end{theorem}
MSS point out that this can be deduced from the following theorem, 
\begin{theorem}[MSS]\label{RV}
 Given independent random vectors $r_1, \cdots, r_k \in \mathbb{C}^n$ with finite support, we have that,
\[\mathbb{P} \left[ \lambda_{k} \left[\sum r_ir_i^* \right]  \geq  \lambda_{k} \mathbb{E}\chi\left[\sum r_ir_i^* \right]\right] > 0.\]
Further, if the outer products all have expectation equal to $\dfrac{I}{m}$, we have that,
\[\mathbb{E}\chi\left[\sum r_ir_i^* \right] = \left(1-D/m\right)^k x^n.\]
\end{theorem}
Theorem (\ref{RISimple}) follows from theorem (\ref{RV})  by sampling the vectors  $v_1, \cdots, v_m$ $k$ times, uniformly, independently and \emph{with} replacement. 

Now suppose, we sampled the vectors $v_1, \cdots, v_m$ a total of $k$ times, uniformly and \emph{without} replacement; Note that this can no longer be modelled using independent random vectors. Any specialization would be of the form.
\[\sum_{i \in S} v_i v_i^{*}, \quad S \subset [m], \,|S| = k.\]
Letting $X$ be the $n \times k$ matrix with columns given by the above vectors, we are interested in the $k$'th eigenvalue of $XX^{*}$. But this equals the smallest eigenvalue of the $k \times k$ matrix $X^{*}X$. This last matrix is a principal submatrix of the $m \times m$ matrix $Y^{*}Y$ where $Y = \left[v_1\mid\cdots\mid v_m\right]$. We conclude that working with principal sub matrices instead of sums of outer products corresponds to sampling vectors uniformly \emph{without} replacement. The machinery of interlacing polynomials still works in this setting and one is able to get slightly better estimates. 

    \subsection{Tightness of bounds}
    
    Lemma (\ref{rrs}) shows that when the polynomial $p$ is real rooted, we have, letting $|\sigma(p)|$ be the size of the smallest interval containing $\sigma(p)$, that for any $c \geq \frac{1}{2}$,
    \[|\sigma \left(p^{(cn)}\right)| \leq 2\sqrt{c-c^2} |\sigma(p)|.\]

   This is sharp: The polynomial $(z^2-1)^m$ shows that one needs to take the derivative at least $\frac{n}{2}$ times where $n = 2m$ to have all the roots migrate inward from the end points. Further,  a simple calculation involving comparing coefficients shows that, 
   \[\sum_{\lambda \in \sigma(p^{(cn)})} \lambda^2 = \dfrac{(n-cn)(n-cn-1)}{n(n-1)} = n(1-c)^2 - \dfrac{nc}{n-1}.\]
   This implies that there is at least one root of modulus at least $\sqrt{1-c}-O(\frac{1}{n})$ and since the roots of $p^{(cn)}$ are symmetric about $0$, the smallest interval containing all the roots of $p^{(cn)}$ contains $[-\sqrt{1-c}+O(\frac{1}{n}), \sqrt{1-c)}-O(\frac{1}{n})]$. We conclude that in the class of real rooted polynomials, which we denote $\mathcal{Q}$ and for any $c \geq \frac{1}{2}$, 
  \[\operatorname{sup}_{p \in \mathcal{Q}} \dfrac{|\sigma \left(p^{(cn)}\right)|}{ |\sigma(p)|} \geq \sqrt{1-c}.\]
  This shows that the upper bound from theorem (\ref{rrs}) is optimal upto a constant. 
   For the complex rooted case, we make an analogous calculation with the polynomial $(z^3-1)^{n}$.
   We have, 
  \[p(z) = z^{3n} - nz^{3n-3} + \cdots,\]
  and
  \[p^{(3cn)}(z) = \binom{3n}{3n(1-c)}\,z^{3(1-c)n} - n\binom{3n-3}{3n(1-c)-3}z^{3n-3} + \cdots.\]
  The polynomial $p^{(3cn)}$ has roots of the form $\{\lambda_i, \lambda_i \omega, \lambda_i \omega^2 :  : 1 \leq i \leq (1-c)n\}$ where the $\lambda_i$ are non-negative reals and we have that,
  \[p^{(3cn)} = \prod_{i = 1}^{(1-c)n} (z^3 - \lambda_i^3).\]
Comparing coefficients, we see that, 
  \[\sum_{i = 1}^{cn} \lambda_i^3 = n \binom{3n-3}{3n(1-c)-3}/ \binom{3n}{3n(1-c)} = n (1-c)^3 + O\left(\frac{1}{n}\right).\]
 The largest of the $\lambda_i$, which we may assume is $\lambda_1$, is therefore at least $(1-c) + O(\frac{1}{n})$. The convex hull of the roots of $p^{(3cn)} $ is the equilateral triangle with vertices $\{\lambda_1, \lambda_1 \omega, \lambda_1 \omega^2\}$ and we see that, 
 \[\dfrac{|\mathcal{K}(p^{(cn)})|}{|\mathcal{K}(p)|} \geq  (1-c)^{4/3} + O\left(\frac{1}{n}\right).\]

 We conclude that, letting $\mathcal{P}$ be the class of all polynomials and working with areas of the convex hulls, 
   \[\operatorname{sup}_{p \in \mathcal{P}} \dfrac{|\mathcal{K}\left(p^{(cn)}\right)|}{ |\mathcal{K}(p)|} \geq (1-c)^{4/3}.\]
   I suspect this can be improved to $O(1-c)$ to match the upper bound.

\section*{Appendix: Proof of majorization lemma \ref{lemma:associativity}}

\newtheorem*{lemma:associativity}{Lemma \ref{lemma:associativity}}\label{xest}
\begin{lemma:associativity}
Suppose $\lambda_i$ for a $i \in [n]$ are a collection of real numbers in $[0,1]$ satisfying 
\[\sum \lambda_i = n \alpha, \qquad \sum \lambda_i^2 = n \beta.\]
Then, for any fixed $b > 1$, the quantity 
\[\varphi = \sum \dfrac{1}{b - \lambda_i},\] satisfies,
\[\varphi \leq \dfrac{ns}{b-1} + \dfrac{nt}{b-x}\]
where $s, t, x$ are given by, 
\[x = \dfrac{\alpha - \beta}{1-\alpha}, \qquad s = \dfrac{\beta - \alpha^2}{1-2\alpha + \beta}, \qquad t = 1- s = \dfrac{(1-\alpha)^2}{1-2\alpha+\beta}.\]
\end{lemma:associativity}

\begin{proof}
We regard the $\lambda_i$ as variables subject to the constraints,
\[\sum \lambda_i = n \alpha, \qquad \sum \lambda_i^2 = n \beta, \qquad \lambda_i \in [0,1],\]
and seek to optimize the quantity,
\[\varphi = \sum \dfrac{1}{b - \lambda_i}.\]
 Suppose $\{\lambda_i\}$ is a maximiser for $\varphi$. And suppose that we have that 
 \[0 \leq \lambda_3 < \lambda_2 < \lambda_1 < 1.\]
 We will derive a contradiction. This will then put strong constraints on the possible values the $\lambda_i$ can take. We can perturb $\lambda_1, \lambda_2, \lambda_3$ so that the two constraints are still satisfied. We will modify them to $\lambda_1 + \delta_1, \lambda_2-\delta_2, \lambda_3+\delta_3$, where $\delta_1, \delta_2, \delta_3$ are positive; For the constraints to be satisfied, we must have,
 \[\lambda_1 + \delta_1 + \lambda_2 -\delta_2 + \lambda_3 + \delta_3 = \lambda_1 + \lambda_2 + \lambda_3,\]
 yielding that,
 \begin{eqnarray}\label{sum}
 \delta_1+\delta_3 = \delta_2.
 \end{eqnarray}
 Also, we must have,
  \[(\lambda_1 + \delta_1)^2 + (\lambda_2 -\delta_2)^2 + (\lambda_3 + \delta_3)^2 = \lambda_1^2 + \lambda_2^2 + \lambda_3^2,\]
   yielding that,
 \[\lambda_2\delta_2 -\lambda_1\delta_1-\lambda_3\delta_3  = \delta_1^2+\delta_2^2+\delta_3^2,\]
 or, using (\ref{sum}), 
 \begin{eqnarray}\label{d1d3}
 (\lambda_2-\lambda_3)\delta_3= (\lambda_1-\lambda_2)\delta_1 +\delta_1^2+\delta_3^2+(\delta_1+\delta_3)^2.
 \end{eqnarray}
 We may solve this for $\delta_1$ for fixed $\delta_3$, with one solution being,  \[\delta_1 = \dfrac{-(\lambda_1-\lambda_2 + 2\delta_3) + \sqrt{(\lambda_1-\lambda_2 + 2\delta_3)^2 + 8\left[(\lambda_2-\lambda_3)\delta_3 - 2\delta_3^2\right]}}{4}.\]
 The discriminant is positive for $\delta_3 < \dfrac{\lambda_1-\lambda_2}{2}$ and in short, we see that under the hypothesis, for sufficiently small and positive $\delta_3$, this can be solved to yield a positive $\delta_1$ with $\delta_1 =O(\delta_3)$. We now claim that for sufficiently small $\delta_1, \delta_2, \delta_3$ (which can be taken to be all positive) satisfying the above constraints, we have that 
\[\dfrac{1}{b-\lambda_1-\delta_1}+\dfrac{1}{b-\lambda_2+\delta_2}+\dfrac{1}{b-\lambda_3-\delta_3} > \dfrac{1}{b-\lambda_1} + \dfrac{1}{b-\lambda_2}+\dfrac{1}{b-\lambda_3}.\]
 We need to show that 
 \[\dfrac{\delta_1}{(b-\lambda_1-\delta_1)(b-\lambda_1)} + \dfrac{\delta_3}{(b-\lambda_3-\delta_3)(b-\lambda_3)} > \dfrac{\delta_2}{(b-\lambda_2+\delta_2)(b-\lambda_2)}.\]
 Taking the $\delta_i$ suitably small, this will follow if we can show that,
 \[\dfrac{\delta_1}{(b-\lambda_1)^2} + \dfrac{\delta_3}{(b-\lambda_3)^2} > \dfrac{\delta_2}{(b-\lambda_2)^2}.\]
 We have that
 \begin{eqnarray*}
 \dfrac{\delta_1}{(b-\lambda_1)^2} + \dfrac{\delta_3}{(b-\lambda_3)^2} - \dfrac{\delta_2}{(b-\lambda_2)^2} &=& \dfrac{\delta_1}{(b-\lambda_1)^2} + \dfrac{\delta_3}{(b-\lambda_3)^2} - \dfrac{\delta_1 + \delta_3}{(b-\lambda_2)^2},
 \end{eqnarray*}
 which equals, 
 \begin{eqnarray*}
  \dfrac{1}{(b-\lambda_2)^2} \left[ \delta_1 (\lambda_1-\lambda_2) \dfrac{2b-\lambda_1-\lambda_2}{(b-\lambda_1)^2} - \delta_3 (\lambda_2-\lambda_3) \dfrac{2b-\lambda_2-\lambda_3}{(b-\lambda_3)^2}\right].
 \end{eqnarray*}
 Using the fact that $\delta_1 = O(\delta_3)$ together with (\ref{d1d3}), we see that, 
\[ (\lambda_2-\lambda_3)\delta_3= (\lambda_1-\lambda_2)\delta_1 + O(\delta_3^2)\]
 Together with the fact that
\begin{eqnarray*}
\dfrac{2b-\lambda_1-\lambda_2}{(b-\lambda_1)^2} - \dfrac{2b-\lambda_2-\lambda_3}{(b-\lambda_3)^2} = \dfrac{1}{b-\lambda_1}-\dfrac{1}{b-\lambda_3} + (b-\lambda_2)\left[\dfrac{1}{(b-\lambda_1)^2}-\dfrac{1}{(b-\lambda_3)^2}\right] > 0
\end{eqnarray*}
 we conclude that the potential function indeed increases.

 This implies that if $\{\lambda_i\}$ is a maximizer, then, the numbers, can take on at most two values apart from $1$. Let us call these values $x$ and $y$, where $x \leq y$ and let the values $ x, y, 1$ be taken on respectively $ ns, nt$ and $n(1-s-t)$ times. These frequencies are integers, but let us relax them further to reals. We have the following constraints, 
 \[sx+ty+1-s-t=\alpha, \qquad sx^2+ty^2+1-s-t =\beta,\]
 yielding that
 \[s = \dfrac{y(1-\alpha)-(\alpha-\beta)}{(1-x)(y-x)}, \qquad t = \dfrac{(\alpha-\beta)-x(1-\alpha)}{(1-y)(y-x)}.\]
 Since $s$ and $t$ cannot be non-negative and must have sum at most $1$, we must have that
 \begin{eqnarray}\label{cons}
 0 \leq  x \leq \dfrac{\alpha-\beta}{1-\alpha} \leq \dfrac{\beta}{\alpha}\leq y \leq 1, \quad \text{ and } \quad  y \geq \dfrac{\alpha x - \beta}{x-\alpha}, 
 \end{eqnarray}

 with no other constraints. We now seek to optimize the potential over $x, y$ in $[0,1)$, the potential being, 
 \begin{eqnarray*}
  \varphi &=& n\left[\dfrac{s}{b-x}+\dfrac{t}{b-y}+\dfrac{1-s-t}{b-1}\right]\\
  &=& \dfrac{n}{b-1}+\dfrac{n}{b-1}\left[ \dfrac{s(1-x)}{b-x} + \dfrac{t(1-y)}{b-y}\right]\\
  &=& \dfrac{n}{b-1}-\dfrac{n}{b-1} \left[ \dfrac{y(1-\alpha)-(\alpha-\beta)}{(b-x)(y-x)} + \dfrac{(\alpha-\beta)-x(1-\alpha)}{(b-y)(y-x)}\right]\\
  &=& \dfrac{n}{b-1}-\dfrac{n(1-\alpha)}{b-1} \left[ \dfrac{1}{b-x} + \dfrac{1}{b-y} - \dfrac{b-\dfrac{\alpha-\beta}{1-\alpha}}{(b-x)(b-y)} \right].
 \end{eqnarray*}

 Recalling that $x \leq \dfrac{\alpha - \beta}{1-\alpha} \leq y$, it is easy to see that this expression decreases in $y$ for fixed $x$ and increases in $x$ for fixed $y$.  The constraint (\ref{cons}) says that for fixed $x$, the variable $y$ can have minimum value $\dfrac{\alpha x - \beta}{x-\alpha}$. We thus have that the maximizer happens for a tuple $(x, y) = (x, \dfrac{\alpha x - \beta}{x-\alpha})$ for some $x \in [0, \dfrac{\alpha-\beta}{1-\alpha}]$.

 We need to minimize 
 \[ \dfrac{1}{b-x} + \dfrac{1}{b-y} - \dfrac{b-\dfrac{\alpha-\beta}{1-\alpha}}{(b-x)(b-y)}  = \dfrac{b - x - y+\dfrac{\alpha-\beta}{1-\alpha}}{b(b-x-y+xy)}\]
 We also have the constraint, $y = \dfrac{\alpha x - \beta}{x-\alpha}$, which can be written as $xy = \alpha(x+y)-\beta$. We therefore need to minimize
 \[\dfrac{b+\dfrac{\alpha-\beta}{1-\alpha} - (x + y)}{\dfrac{b-\beta}{1-\alpha}-(x+y)} = 1-\dfrac{b}{\dfrac{b-\beta}{1-\alpha}-(x+y)}. \]
 
 It is easy to see that this is minimized when $x+y$ is as large as possible, but since $x+y = x+\alpha+\dfrac{\beta - \alpha^2}{\alpha - x}$, this maximum value of the potential is attained at the maximum possible value of $x$, namely $x  = \dfrac{\alpha-\beta}{1-\alpha}$. We will then have that,
  \[y = \dfrac{\alpha\left(\dfrac{\alpha-\beta}{1-\alpha}\right) - \beta}{\dfrac{\alpha-\beta}{1-\alpha} - \alpha} = 1.\]

 A simple calculation now shows that the potential is exactly the expression in the statement of the lemma, namely,
 \[\varphi = \dfrac{ns}{b-1} + \dfrac{nt}{b-x}, \]
 where, 
 \[x = \dfrac{\alpha - \beta}{1-\alpha}, \qquad s = \dfrac{\beta - \alpha^2}{1-2\alpha + \beta}, \qquad t = 1- s = \dfrac{(1-\alpha)^2}{1-2\alpha+\beta}.\]

\end{proof}

\end{document}